%% file: AutFix-11.tex
\documentclass[11pt]{amsart}
\usepackage{amsfonts, amsmath, amssymb, amscd, amsthm, graphicx}
\usepackage{epsfig, color}
\usepackage{hyperref}

\makeatletter
\def\blfootnote{\xdef\@thefnmark{}\@footnotetext}
\makeatother

\newtheorem{thm}{Theorem}[section]
\newtheorem{cor}[thm]{Corollary}
\newtheorem{lem}[thm]{Lemma}

\theoremstyle{definition}
\newtheorem{defn}[thm]{Definition}
\theoremstyle{remark}
\newtheorem{rem}[thm]{Remark}

\newfont{\eufm}{eufm10}

\newcommand{\G}{\Gamma (G, X\cup \mathcal H)}
\newcommand{\dxh}{{\rm d}_{X\cup\mathcal H}}
\newcommand{\dx}{{\rm d}_X}
\newcommand{\lx}{l_X}
\newcommand{\Hl}{\{ H_\lambda \} _{\lambda \in \Lambda } }
\newcommand{\e}{\varepsilon }
\renewcommand{\phi}{\varphi}
\newcommand{\F }{{\rm Fix} (\phi )}
\newcommand{\A }{{\rm Aut} (G)}
\newcommand{\Lab}{{\bf Lab}}

\newcommand{\N}{\mathbb{N}}

\begin{document}

\title{Fixed subgroups of automorphisms of relatively hyperbolic groups}

\author{Ashot Minasyan}
\address[Ashot Minasyan]{School of Mathematics,
University of Southampton, Highfield, Southampton, SO17 1BJ, United
Kingdom.}  \email{aminasyan@gmail.com}

\author{Denis Osin}
\address[Denis Osin]{Department of Mathematics, Vanderbilt University, Nashville TN 37240, USA.}
\email{denis.osin@gmail.com}
\thanks{The second author was supported by the
NSF grants DMS-0605093, DMS-1006345, and by the RFBR grant 05-01-00895.}
\date{}

\begin{abstract}
Let $G$ be a finitely generated relatively hyperbolic group. We show that if no peripheral subgroup of $G$ is
hyperbolic relative to a collection of proper subgroups, then the fixed subgroup of every automorphism of $G$
is relatively quasiconvex. It follows that the fixed subgroup is itself relatively hyperbolic with respect to
a natural family of peripheral subgroups. If all peripheral subgroups of $G$ are slender (respectively,
slender and coherent), our result implies that the fixed subgroup of every automorphism of $G$ is
finitely generated (respectively, finitely presented). In particular, this happens when $G$ is a limit group,
and thus for any $\phi \in \A$, $\F$ is a limit subgroup of $G$.
\end{abstract}



\maketitle
%%%%%%%%%%%%%%%%%%%%%%%%%%%%%%%%%%%%%%%%%%%%%%%%%%%%%%%%%%%%%%%%%%%%%%%%%%%%%%%%%%%%%%%%%%%%%%%%%%%%%%%%%%%%%%%%%%%%%%

\section{Introduction}

%%%%%%%%%%%%%%%%%%%%%%%%%%%%%%%%%%%%%%%%%%%%%%%%%%%%%%%%%%%%%%%%%%%%%%%%%%%%%%%%%%%%%%%%%%%%%%%%%%%%%%%%%%%%%%%%%%%%%%

Given a group $G$ and an automorphism $\phi \in \A$, let $\F$ denote the {\it fixed subgroup} of $\phi $, i.e.,
$$
\F =\{ g\in G\mid \phi (g)=g\} .
$$
Gersten \cite{Ger} proved that if $G$ is a finitely generated free group, then $\F $ is finitely generated for every $\phi \in \A$. Collins and Turner \cite{CT} generalized this result by showing that $\F $ has a finite Kuro\v s decomposition provided $G$ is a finite free product of freely indecomposable groups. Another generalization of Gersten's theorem was found by Neumann \cite{Neumann}, who showed that for every word
hyperbolic group $G$ and every $\phi \in \A$, $\F $ is quasiconvex in $G$.

In this paper we study fixed subgroups of automorphisms of a more general class of groups, which includes hyperbolic groups as well as finitely generated free products of freely indecomposable groups. More precisely, we deal with finitely generated groups hyperbolic relative to NRH subgroups. Recall that a nontrivial group $H$ is called {\it non-relatively hyperbolic} (or {\it NRH}) if $H$ is not hyperbolic relative to any collection of proper subgroups. The class of NRH groups includes many examples of interest. Below we list just some of them.

\begin{enumerate}
\item[(1)] Unconstricted groups (defined by C. Dru\c tu and M. Sapir in \cite{DS}).
Recall that a finitely generated group $H$ is {\it  unconstricted} if
some asymptotic cone of $H$ does not have cut points. Examples of
unconstricted groups include direct products of infinite groups,
non-virtually cyclic groups satisfying a nontrivial law (e.g., solvable groups and groups of
finite exponent) \cite{DS},  many lattices in higher rank semi-simple groups \cite{DMS}, etc.

\item[(2)] Suppose that a group $G$ is generated by a set $X$ consisting of elements of infinite order. The corresponding {\it commutativity graph} has $X$ as the set of vertices, two of which are joined by an edge if the corresponding elements commute. Assume that some adjacent pair of vertices generates ${\mathbb{Z}}\times{\mathbb{Z}}$. Then $G$ is NRH \cite{AAS}. For example, this class includes many constricted groups such as $Out(F_n)$ for $n\ge 3$, all but finitely many mapping class groups, and freely indecomposable right angled Artin groups.

\item[(3)] Non-virtually cyclic groups with infinite center \cite[Lemma 10.2]{Min-RAAG}.

\item[(4)] Non-virtually cyclic groups which do not contain non-abelian free subgroups (\cite[Prop. 6.5]{DS}).
In particular, non-virtually cyclic amenable groups as well as various `monsters'.
\end{enumerate}

In many cases when peripheral subgroups are relatively hyperbolic themselves, we can still get an NRH peripheral structure in the following way. Suppose that $G$ is hyperbolic relative to $\{ H_1, \ldots , H_m\}$ and every $H_i$ is hyperbolic relative to proper subgroups $\{ K_{i1}, \ldots , K_{in_i}\}$. Then $G$ is hyperbolic relative to $\mathcal K=\{K_{ij}\mid i=1,\ldots , m,\, j=1, \ldots , n_i\} $ (\cite[Cor. 1.14]{DS}). We exclude trivial subgroups from $\mathcal K$ and if some of the subgroups from $\mathcal K$ are hyperbolic relative to proper subgroups, we repeat this step again. We say that the process \textit{terminates}, if after some step we obtain a (possibly empty) collection of NRH peripheral subgroups. 

Note that the above process may not terminate even for hyperbolic groups. Recall that any hyperbolic group $G$ is  hyperbolic relative to any quasiconvex malnormal subgroup \cite{Bow}. Thus any infinite sequence of quasiconvex malnormal subgroups $G\gneqq H \gneqq K \gneqq ...$ leads to an infinite process. However, in this case there also exists an obvious process which does stop as $G$ is hyperbolic relative to the empty collection of subgroups.  Behrstock, Dru\c tu and Mosher showed that there exists a finitely generated group for which \textit{no}
 such a process terminates \cite[Proposition 6.3]{BDM}). The question of whether for every \textit{finitely presented} relatively hyperbolic group there exists a terminating process is still open.

Our main result is the following.

\begin{thm} \label{thm:main}
Let $G$ be a finitely generated group which is hyperbolic relative to a family of
NRH subgroups. Then for every $\phi \in \A$, the fixed subgroup  $\F $ is
relatively quasiconvex in $G$.
\end{thm}

Since the intersection of any two relatively quasiconvex subgroups is relatively
quasiconvex \cite[Prop. 4.18]{Osi06},
the same result holds for any finite collection of automorphisms.

In order to prove Theorem \ref{thm:main}
we first show that any automorphism $\phi$ of a group $G$, that is
hyperbolic relative to a collection of NRH subgroups, \textit{respects the peripheral structure}
(see Definition \ref{df:resp_per_str}). Although this observation is quite elementary,
it plays an important role
in our paper. It allows us to conclude that $\phi$ induces a quasiisometry  of the relative Cayley
graph and to use the geometric machinery of relatively
hyperbolic groups partially developed in \cite{Osi06}. In fact, our proof of Theorem \ref{thm:main} would
remain valid if instead of requiring the peripheral subgroups of $G$ to be NRH one demanded the
automorphism $\phi$ to respect some peripheral structure on $G$.
It would be interesting to see whether the conclusion of Theorem \ref{thm:main} holds in general,
without any of these two requirements.

Note that relative quasiconvexity of a subgroup is independent of the choice of the finite generating set
for the group $G$ \cite[Prop. 4.10]{Osi06}, but may, in general, depend on the selection of
the family of peripheral subgroups. Nevertheless, relatively quasiconvex subgroups are well-behaved and
have many good properties. For instance, C. Hruska \cite{Hr} proved that relatively quasiconvex
subgroups of relatively hyperbolic groups are themselves relatively hyperbolic
with a natural induced peripheral structure. This allows us to obtain some results
about the algebraic structure of fixed subgroups.

\begin{cor}\label{cor1}
Assume that $G$ is a finitely generated group which is hyperbolic relative to a family of NRH subgroups
and $\phi \in \A$. Let $\mathcal P$ be the set of all conjugates of peripheral
subgroups of $G$ and let
$$\mathcal P ^\phi =\{ P \cap \F \mid P\in \mathcal P\;{\rm and}\; |P\cap \F|=\infty \}.$$
Then the action of $\F $ on $\mathcal P ^\phi $ by conjugation has finitely many orbits
and $\F $ is hyperbolic relative to representatives of these orbits.
In particular, if $\mathcal P^\phi$ is empty, then $\F $ is finitely generated and word hyperbolic.
\end{cor}

Recall that a group $G$ is called {\it slender} if every subgroup of $G$ is finitely generated and $G$ is called {\it coherent} if every finitely generated subgroup of $G$ is finitely presented. Since a group hyperbolic relative to a finite family of finitely generated (respectively, finitely presented) peripheral subgroups is itself finitely generated (respectively, finitely presented), the next result easily follows from Corollary \ref{cor1}.

\begin{cor}\label{cor:all_fg}
If a finitely generated group $G$ is hyperbolic relative to slender subgroups, then for every $\phi \in Aut (G)$, $\F $ is finitely generated. If, in addition, all peripheral subgroups of $G$ are coherent, then $\F$ is finitely presented. In particular, the latter conclusion holds for finitely generated relatively hyperbolic groups with virtually polycyclic peripheral subgroups.
\end{cor}

One particular application of the above corollary shows that for any automorphism $\phi$ of a limit group $G$,
$\F$ is finitely generated, and thus is a limit group itself. Indeed, Dahmani \cite{Dahmani} and, independently, Alibegovi\'c \cite{Alib} proved that any limit group is hyperbolic relative to the
collection of representatives of conjugacy classes of maximal abelian non-cyclic subgroups.

Neumann's original motivation for showing that $\F$ is quasiconvex when $G$ is a hyperbolic group \cite{Neumann} was the result of S. Gersten and H. Short \cite[Thm. 2.2]{G-S}
that, for a regular language $\mathcal L$ on a group $G$, a subgroup $H \le G$ is $\mathcal L$-rational if and only if $H$ is $\mathcal L$-quasiconvex. This result can be combined with
Gromov's theorem  \cite[8.5]{Gro}, claiming that in any hyperbolic group the set of all geodesic words forms a regular language ${\mathcal L}_{geod}$, to conclude that $\F$ is rational
with respect to ${\mathcal L}_{geod}$. It is also known that rationality and quasiconvexity of a subset of a hyperbolic group are independent of the choice of an automatic structure on it
\cite{N-S}. 

In a relatively hyperbolic group the situation is more complicated. For example, let $G=\langle a,b,t \, \| \, ab=ba\rangle$ be  the free product of a free abelian group of rank $2$
with an infinite cyclic group, and let $\phi \in Aut(G)$ be the automorphism interchanging $a$ and $b$ and sending $t$ to $t^{-1}$. Then $G$ is hyperbolic relative to $\langle a,b \rangle$
and $\F=\langle ab \rangle$. It is easy to see that
$\mathcal{L}_0=\{a^m b^n \mid m,n \in \mathbb{Z} \}$ is a regular language on $\langle a,b \rangle$, which can be naturally extended (using normal forms in free products) to an
automatic language $\mathcal{L}_1$ on $G$. However, in this case  $\F$ is not $\mathcal{L}_1$-quasiconvex, and thus $\F$ will not be $\mathcal{L}_1$-rational.

Nonetheless, in the case when all of the peripheral subgroups of a relatively hyperbolic group $G$ are abelian and $\phi \in Aut(G)$, Theorem  \ref{thm:main} can be combined with a result of
D. Rebbechi \cite[Thm. 9.1]{Reb} in order to conclude that $\F$ is biautomatic.

The paper is structured as follows. In Section \ref{sec:prelim} we include the necessary background on relatively hyperbolic groups and relatively quasiconvex subgroups; in Section \ref{sec:tech_lemmas}
we give several auxiliary definitions and prove a number of technical results, which will be employed in our proof of the main result in
Section \ref{sec:main-proof}.

{\bf Acknowledgment}.
The results of this paper were obtained when the second author was stuck in the UK due to the eruption of the Eyjafjallaj\"okull volcano in April 2010. We would like to thank Eyjafjallaj\"okull for providing us with a great opportunity to work together. We also thank the anonymous referee for helpful remarks and suggestions.

%%%%%%%%%%%%%%%%%%%%%%%%%%%%%%%%%%%%%%%%%%%%%%%%%%%%%%%%%%%%%%%%%%%%%%%%%%%%%%%%%%%%%%%%%%%%%%%%%%%%%%%%%%%%%%%%%%%%%%

\section{Preliminaries}\label{sec:prelim}

%%%%%%%%%%%%%%%%%%%%%%%%%%%%%%%%%%%%%%%%%%%%%%%%%%%%%%%%%%%%%%%%%%%%%%%%%%%%%%%%%%%%%%%%%%%%%%%%%%%%%%%%%%%%%%%%%%%%%%

\noindent{\bf Notation.} Given a group $G$ generated by a subset $S\subseteq G$, we denote by $\Gamma (G,S) $ the Cayley graph of $G$ with respect to $S$ and by $|g|_S$ the word length of an element $g\in G$. We always assume that generating sets are symmetrized, i.e., $S^{-1}=S$. If $p$ is a (simplicial) path in $\Gamma (G,S)$,
$\Lab (p)$ denotes its label, $l(p)$ denotes its length, $p_-$ and $p_+$ denote its starting and ending vertex
respectively. The notation $p^{-1}$ will be used for the path in $\Gamma (G,S) $ obtained by traversing
$p$ backwards. For a word $W$, written in the alphabet $S$, $\|W\|$ will denote its length.
For two  words $U$ and $V$ we shall write $U \equiv V$ to denote the letter-by-letter equality between them.

\paragraph{\bf Relatively hyperbolic groups.} In this paper we use
the notion of relative hyperbolicity which is sometimes called
strong relative hyperbolicity and goes back to Gromov \cite{Gro}.
There are many equivalent definitions of (strongly) relatively
hyperbolic group. We briefly recall one of them and refer the reader to \cite{Bow,DS,Farb,Hr,Osi06} for details.

Let $G$ be a group, let $\Hl $ be a collection of pairwise distinct subgroups of $G$,
and let $X$ be a subset of $G$. We say that $X$ is a {\it relative generating set of
$G$ with respect to $\Hl $} if $G$ is generated by $X$ together with
the union of all $H_\lambda $. (In what follows we always assume $X$
to be symmetric, i.e., $X=X^{-1}$.)
In this situation the group $G$ can be regarded as
a quotient group of the free product
\begin{equation}
F=\left( \ast _{\lambda\in \Lambda } H_\lambda  \right) \ast F(X),
\label{F}
\end{equation}
where $F(X)$ is the free group with the basis $X$. If the kernel of
the natural homomorphism $F\to G$ is the normal closure of a subset
$\mathcal R$ in the group $F$, we say that $G$ has {\it relative
presentation}
\begin{equation}\label{G}
\langle X,\; H_\lambda, \lambda\in \Lambda \; | \; \mathcal R
\rangle .
\end{equation}
If $|X|<\infty $ and $|\mathcal R|<\infty $, the
relative presentation (\ref{G}) is said to be {\it finite} and the
group $G$ is said to be {\it finitely presented relative to the
collection of subgroups $\Hl $.}

Set
\begin{equation}\label{H}
\mathcal H=\bigsqcup\limits_{\lambda\in \Lambda} (H_\lambda
\setminus \{ 1\} ) .
\end{equation}
Given a word $W$ in the alphabet $X\cup \mathcal H$ such that $W$
represents $1$ in $G$, there exists an expression
\begin{equation}
W\stackrel{F}{=} \prod\limits_{i=1}^k f_i^{-1}R_i^{\pm 1}f_i \label{prod}
\end{equation}
with the equality in the group $F$, where $R_i\in \mathcal R$ and
$f_i\in F $ for $i=1, \ldots , k$. The smallest possible number
$k$ in a representation of the form (\ref{prod}) is called the
{\it relative area} of $W$ and is denoted by $Area^{rel}(W)$.

\begin{defn}[\bf Relatively hyperbolic groups]\label{def:rel_hyp_gp}
A group $G$ is {\it hyperbolic relative to a collection of
subgroups} $\Hl $, called {\it peripheral subgroups}, if $G$ is finitely presented relative to $\Hl $
and there is a constant $C>0$ such that for any word $W$ in $X\cup
\mathcal H$ representing the identity in $G$, we have
\begin{equation}\label{isop}
Area^{rel}
(W)\le C\| W\| .
\end{equation}
\end{defn}

This definition is independent of the choice of the finite
generating set $X$ and the finite set $\mathcal R$ in (\ref{G})
(see \cite{Osi06}). In particular, $G$ is an ordinary (Gromov) {\it hyperbolic group} if $G$ is hyperbolic relative to the {empty family of peripheral subgroups}.

\begin{rem}\label{rem:parab_fg->gp_fg}
Note that, by definition,  if $|\Lambda|<\infty$ and each subgroup $H_\lambda$, $\lambda \in \Lambda$,
is finitely generated [finitely presented], then $G$ is also finitely generated [resp., finitely presented].
\end{rem}

Let $G$ be a group generated by a finite set $X \subset G$ and let $H \le G$ be a subgroup generated by a
finite set $A \subset H$. Recall, that $H$ is said to be \textit{undistorted} in $G$, if there exists $C >0$
such that for every $h \in H$ one has $|h|_A \le C |h|_X$.

In general, the relatively hyperbolic group $G$ does not have to be finitely generated, and the collection of
peripheral subgroups $\Hl$ could be infinite. However, the second author proved the following:

\begin{lem}[\cite{Osi06}, Thm. 1.1 and Lemma 5.4]\label{finLambda}
Let $G$ be a finitely generated relatively hyperbolic group. Then the collection of
peripheral subgroups is finite, every peripheral subgroup is finitely generated and undistorted in $G$.
\end{lem}

\begin{lem}[\cite{Osi06}, Thm. 1.4]\label{maln}
Let $G$ be a group hyperbolic relative to a collection of subgroups $\Hl $. Then the following conditions hold.
\begin{enumerate}
\item For every $\lambda, \mu \in \Lambda $, $\lambda \ne \mu $, and every $g\in G$, we have $|H_\lambda \cap H_\mu^g |<\infty $.

\item For every $\lambda \in \Lambda $ and $g\in G\setminus H_\lambda $, we have $|H_\lambda \cap H_\lambda ^g|<\infty $.
\end{enumerate}
\end{lem}

\paragraph{\bf Relatively quasiconvex and undistorted subgroups.} The following definition was suggested in \cite{Osi06}.

\begin{defn}[\bf Relatively quasiconvex subgroups]
Let $G$ be a group generated by a finite set $X$ and hyperbolic relative to a family
of subgroups $\Hl$. A subgroup $H \le G$ is called relatively
quasiconvex with respect to $\Hl$ (or simply relatively quasiconvex when
the collection $\Hl$ is fixed) if there exists a constant $\sigma \ge 0$ such that
the following condition holds. For any $g,h \in H$ and any geodesic path $p$ from $g$ to $h$
in $\G$, each vertex $v$ of $p$ satisfies $\dx(v,H) \le \sigma$.
\end{defn}

We will need two results about relatively quasiconvex subgroups.
The first one is established in \cite[Thm. 9.1]{Hr}.% and Lemma \ref{finLambda}.

\begin{lem}\label{hr}
Let $G$ be a finitely generated relatively hyperbolic group, $K$ a relatively quasiconvex subgroup of $G$. Let $\mathcal P$ be the set of all conjugates of peripheral subgroups of $G$ and let $$\mathcal Q =\{ P \cap K \mid P\in \mathcal P\;{\rm and}\; |P\cap K|=\infty \}.$$  Then the action of $K $ on $\mathcal Q $ by conjugation has finitely many orbits and $K $ is hyperbolic relative to representatives of these orbits.
\end{lem}

In \cite[Thm. 1.8]{DS} Dru\c tu and Sapir showed that the conclusion of Lemma \ref{hr}
holds for every undistorted subgroup $K\le G$.
Later Hruska proved the following in \cite[Thm. 1.5]{Hr}:

\begin{lem}\label{ds}
Let $G$ be a finitely generated group hyperbolic with respect to a collection of subgroups $\Hl$ and let $K$ be a finitely generated undistorted subgroup of $G$. Then $K$ is relatively quasiconvex. In particular, the conclusion of the previous lemma holds for $K$.
\end{lem}

Note that relative quasiconvexity does not, in general, imply that the subgroup is undistorted
(indeed, by definition, any subgroup $K$ of a peripheral subgroup $H$ is relatively quasiconvex; however,
$K$ may be distorted in $H$, and hence in $G$).

\paragraph{\bf Components.} Let $G$ be a group hyperbolic relative to a family of subgroups $\Hl$.
We recall some auxiliary terminology introduced in \cite{Osi06}, which
plays an important role in our paper.

\begin{defn}[\bf Components]
Let $q$ be a path in the Cayley graph $\G $. A (non-trivial)
subpath $p$ of $q$ is called an {\it $H_\lambda $-component} (or simply a {\it component}),
if the label of $p$ is a word in the
alphabet $H_\lambda\setminus \{ 1\} $,  for some $\lambda \in \Lambda$, and
$p$ is not contained in a longer subpath of $q$ with this property.
Two $H_\lambda $-components $p_1, p_2$
of paths $q_1$, $q_2$ (respectively) in $\G $ are called {\it connected} if there exists a
path $c$ in $\G $ that connects some vertex of $p_1$ to some vertex
of $p_2$, and ${\Lab (c)}$ is a word consisting of letters from
$H_\lambda\setminus\{ 1\} $. In algebraic terms this means that all
vertices of $p_1$ and $p_2$ belong to the same coset $gH_\lambda $
for a certain $g\in G$. Note that we can always assume that $c$ has
length at most $1$, as every non-trivial element of $H_\lambda
\setminus\{ 1\} $ is included in the set of generators. A component of a path $p$ is {\it isolated}
if it is not connected with any other component of $p$.
\end{defn}

In what follows, let $G$ be a group hyperbolic relative to a collection of subgroups $\Hl $ and generated by a finite set $X$. Note that $\Lambda $ is finite in this case and every $H_\lambda $ is finitely generated \cite[Theorem 1.1]{Osi06}.

Let $\varkappa \ge 1$ and $c \ge 0$ be real numbers and let $p$ be a path in $\G$. Recall that $p$ is said to be
$(\varkappa,c)$-\textit{quasigeodesic} if for any subpath $q$ of $p$ we have $l(q) \le \varkappa \dxh(q_-,q_+) +c$.
It is not difficult to see that a path that is a concatenation of a geodesic path with a path
of length at most $c$ is $(1,2c)$-quasigeodesic.

Given a path $p$ in $\G $ we denote by $\lx (p)$ the $X$-length of the element represented by the label of $p$;
in other words, $\lx (p)=\dx (p_-, p_+)$.  Recall that a path $p$ in $\G $ is called a {\it path without backtracking} if for any $\lambda\in \Lambda $, every $H_\lambda$--component of $p$ is isolated.
Evidently any geodesic path in $\G$ is without backtracking.
The following is a reformulation of Farb's Bounded Coset Penetration property (cf. \cite{Farb})
in terms of the relative Cayley graph (see \cite[Theorem 3.23]{Osi06}).

\begin{lem}\label{BCP}
For any $\varkappa \ge 1$, $c\ge 0$ and $k \ge 0$,
there exists a constant $\e =\e(\varkappa , c,k) \ge 0$ such that the following
conditions hold. Let $p$, $q$ be $(\varkappa , c)$-quasigeodesics
without backtracking in $\G $ such that $\dx(p_-,q_-) \le k$ and $\dx(p_+,q_+) \le k$.
\begin{enumerate}
\item[(a)] Suppose that for some $\lambda \in \Lambda $, $s$ is an $H_\lambda $-component of $p$
such that $\lx (s) \ge \e$; then there exists an
$H_\lambda$--component $t$ of $q$ such that $t$ is connected to $s$.

\item[(b)] Suppose that for some $\lambda \in \Lambda $, $s$ and $t$ are connected
$H_\lambda $-components of $p$ and $q$ respectively. Then $\dx (s_-,t_-)\le \e $ and $\dx (s_+, t_+)\le \e$.
\end{enumerate}
\end{lem}

A vertex of a path $p$ in $\G$ is \textit{phase} if it is not
an inner vertex of some component of $p$. Observe that every vertex of a geodesic segment
is phase, because all components consist of single edges.
It is well known that in a hyperbolic group quasigeodesics with same
endpoints are uniformly close to each other. An analogue of this statement for relatively hyperbolic groups
was established in \cite[Prop. 3.15]{Osi06}:

\begin{lem}\label{lem:qg-close} For any $\varkappa \ge 1$, $c \ge 0$ and $k \ge 0$ there exists a
constant $\nu = \nu(\varkappa, c,k) \ge 0$ having the following property. Let $p$ and $q$ be two
$(\varkappa,c)$-quasigeodesic paths in $\G$ such that $\dx(p_-,q_-)\le k$, $\dx(p_+,q_+)\le k$
and $p$ is without backtracking.
Then for any  phase vertex $u$ of $p$ there exists a phase vertex $v$ of $q$ such that $\dx(u,v) \le \nu$.
\end{lem}

%The next lemma follows immediately from \cite[Proposition 3.2]{Osi07}.

%\begin{lem}\label{ngon}
%There exists a constant $D$ such that for any geodesic $n$-gon $P$ in $\G $
%and any set of isolated components $c_1, \ldots , c_k$ of $P$, we have $$\sum\limits_{i=1}^k \lx (c_i) \le D(n+k).$$
%\end{lem}

%%%%%%%%%%%%%%%%%%%%%%%%%%%%%%%%%%%%%%%%%%%%%%%%%%%%%%%%%%%%%%%%%%%%%%%%%%%%%%%%%%%%%%%%%%%%%%%%%%%%%%%%%%%%%%%%%%%%%

\section{Technical lemmas}\label{sec:tech_lemmas}

%%%%%%%%%%%%%%%%%%%%%%%%%%%%%%%%%%%%%%%%%%%%%%%%%%%%%%%%%%%%%%%%%%%%%%%%%%%%%%%%%%%%%%%%%%%%%%%%%%%%%%%%%%%%%%%%%%%%%
We start with the following definition.

\begin{defn}[\bf Respecting peripheral structure]\label{df:resp_per_str}
Let $G$ be a group hyperbolic relative to a family of peripheral
subgroups $\Hl$ and let $\phi \in \A$. We will say that $\phi$ \textit{respects the peripheral structure}
of $G$ if for every $\lambda \in \Lambda$ there is $\lambda' \in \Lambda$ such that $\phi(H_\lambda)$
is a conjugate of $H_\lambda'$ in $G$.
\end{defn}

Throughout the rest of the paper $G$ will denote a group generated by a finite set $X$ and hyperbolic
relative to a collection of NRH subgroups $\Hl $. In particular, all peripheral subgroups of $G$
are infinite. Note also that $|\Lambda |<\infty$ by Lemma \ref{finLambda}.

\begin{lem}\label{fixps}
With the above assumptions on $G$, every $\phi\in\A$ respects the peripheral
structure of $G$, i.e., for each $\lambda \in \Lambda$ there is a unique $\lambda ' \in \Lambda$ such that $\phi(H_\lambda)$ is a conjugate of $H_{\lambda'}$ in $G$. Moreover, the map $\Lambda \to \Lambda$,
$\lambda \mapsto \lambda '$  is a bijection.
\end{lem}

\begin{proof}
By Lemma \ref{finLambda} every $H_\lambda $ is finitely generated and undistorted in $G$. Hence so is
$\phi (H_\lambda)$ (because an automorphism is always a quasiisometry when the group is equipped with a
word metric given by some finite generating set). By Lemma \ref{ds}, $\phi (H_\lambda )$ is relatively hyperbolic and each of its peripheral subgroups is an intersection of $\phi (H_\lambda )$ with a conjugate
of some $H_{\mu} $. Since $\phi (H_\lambda )\cong H_\lambda $ is an NRH group, it can be hyperbolic
only relative to itself. Therefore $\phi (H_\lambda )\le g^{-1}H_{\lambda '} g$ for some $g\in G$ and
${\lambda '} \in \Lambda $. If this inclusion is proper, then $\phi^{-1} (g^{-1}H_{\lambda '} g)$ properly contains
$H_\lambda $. Applying Lemma \ref{ds} one more time, we obtain that $\phi^{-1} (g^{-1}H_{\lambda '} g)$
is relatively hyperbolic with respect to a collection of subgroups $\mathcal K$ containing $H_\lambda $.
Since %$|H_\lambda |=\infty $ and
$\phi^{-1} (g^{-1}H_{\lambda '} g) \cong H_{\lambda '}$ is NRH, this is again impossible.
Hence $\phi (H_\lambda )=g^{-1}H_{\lambda '} g$. If $\phi (H_\lambda )$ is also conjugate to $H_{\lambda ''}$ for some $\lambda '' \in \Lambda $, then ${\lambda ''}={\lambda '}$ by Lemma \ref{maln} as all peripheral subgroups are infinite. Repeating the same arguments for $\phi^{-1}$ we obtain that $\lambda \mapsto \lambda '$  is injective (and, hence, bijective)
on $\Lambda $.
\end{proof}

From now on we fix an automorphism $\phi \in \A$.
For each $\lambda \in \Lambda$ fix $f_\lambda \in G$ and $\lambda' \in \Lambda$ so that
$\phi (H_\lambda )=f_\lambda ^{-1}H_{\lambda ^\prime} f_\lambda $.
Since $|\Lambda|<\infty$ and relative hyperbolicity is independent of the choice of the finite
generating set $X$, we can further assume that $f_\lambda^{\pm 1} \in X$  for every $\lambda \in \Lambda $.
Finally we set
\begin{equation}\label{eqS}
S=\max\limits_{x\in X} \{ |\phi (x)|_X, |\phi ^{-1}(x)|_X\}.
\end{equation}

\begin{defn}[{\bf Image of a path}] \label{def:image_of_path}
For every $x\in X$ we fix a shortest word $W_x$ in the alphabet $X$ that represents $\phi (x)$ in $G$.
Let $e$ be an edge of $\G $
labelled by some $g\in X\cup\mathcal H$.
By $\phi (e)$ we denote the path from $\phi (e_-)$ to $\phi (e_+)$ constructed as follows. If $g\in X$, we
define $\phi (e)$ to be the path with label $W_x$. If $g \in H_\lambda \setminus \{ 1\}$ for some $\lambda \in \Lambda $,
$\phi (g)=f_\lambda ^{-1}hf_\lambda $, where $h\in H_{\lambda ^\prime} \setminus \{ 1\}$. In this case we let
$\phi (e)$ to be the path of length $3$ with $\Lab (\phi(e)) \equiv f_\lambda ^{-1}hf_\lambda$.
Hence the middle edge of $\phi(e)$ will be its $H_{\lambda'}$-component; we will call it
the {\it companion} of $e$ and denote by $e_\phi$. Given a path $p=e_1\cdots e_k$ in $\G $, where $e_1, \ldots , e_k$
are edges of $p$, the path $\phi (e_1)\cdots \phi (e_k)$ will be called the {\it image} of $p$, and will be denoted by $\phi (p)$. Note that $\phi(p)_-=\phi(p_-)$ and $\phi(p)_+=\phi(p_+)$.
\end{defn}

Since every component of a geodesic path in $\G$ consists of a single edge, Definition \ref{def:image_of_path} together with Lemma \ref{fixps} and the fact that
$f_\lambda^{\pm 1} \in X$, for all $\lambda \in \Lambda$, easily imply the following

\begin{rem} \label{rem:comp_of_im_of_geod} For a geodesic path $p$ in $\G$, each component of $\phi(p)$ consists of a single edge.
\end{rem}

\begin{lem}\label{Xlength}
(a) Let $e$ be an edge of $\G $ labelled by a letter from $\mathcal H$. Then
$\lx (e_\phi )\le S\lx (e) +2$ and $\lx (e)\le S(\lx (e_\phi) +2)$.

\noindent (b) Let $e, f$ be edges of $\G $ labelled by letters from $\mathcal{H}$.
Then $e$ and $f$ are connected (i.e., there is $\lambda \in \Lambda$ such that
$\Lab(e),\Lab(f) \in H_\lambda \setminus\{1\}$, and vertices of $e$ and $f$
belong to the same left coset of $H_\lambda $)
if and only if $e_\phi $ and $f_\phi $ are connected.
\end{lem}

\begin{proof}
(a)\; Recall that $|f_\lambda|_X \le 1 $ for every $\lambda\in \Lambda$. Using the triangle inequality
and \eqref{eqS} we obtain $\lx (e_\phi )\le \lx (\phi (e)) +2 \le S\lx (e) +2$. Similarly
$ \lx (e)\le S\lx (\phi (e)) \le S(\lx (e_\phi) +2).$

(b)\; Let $x=e_-$, $y=f_-$, $x' =(e_\phi )_-$, $y' =(f_\phi )_-$. If $e$ and $f$ are connected,
then $\Lab(e),\Lab(f) \in H_\lambda \setminus\{1\}$, for some $\lambda \in \Lambda$, and
$x^{-1}y=h\in H_\lambda $. Then $\Lab(e_\phi),\Lab(f_\phi) \in H_{\lambda'} \setminus\{1\}$ and
$\phi(h) \in f_\lambda^{-1} H_{\lambda'} f_\lambda$.
Clearly $x' = \phi (x)f_\lambda ^{-1}$ and $y' = \phi (y)f_\lambda ^{-1}$.
Therefore $(x')^{-1}y' = f_\lambda \phi (x^{-1}y)f_\lambda ^{-1}=
f_\lambda \phi (h) f_\lambda ^{-1} \in H_{\lambda'} $.

Conversely, suppose that $e_\phi$ and $f_\phi$ are connected, i.e.,
$\Lab(e_\phi), \Lab(f_{\phi}) \in H_\mu \setminus \{1\}$, for some $\mu \in \Lambda$ and
$(x')^{-1}y' \in H_\mu$. By Lemma \ref{fixps}, there is a unique $\lambda \in \Lambda$
such that $\mu=\lambda'$. Thus $\Lab(e),\Lab(f) \in H_\lambda \setminus\{1\}$ and
$x' = \phi (x)f_\lambda ^{-1}$, $y' = \phi (y)f_\lambda ^{-1}$. Consequently
$x^{-1}y=\phi^{-1}(f_{\lambda}^{-1} (x')^{-1}y'f_\lambda) \in H_\lambda$, implying that $e$ and $f$
are connected as well.
\end{proof}

\begin{lem}\label{qgimage}
(a) Suppose that $p$ is a path without backtracking in $\G $ such that every component of $p$ is an edge. Then $\phi (p)$ is a path without backtracking.

\noindent (b) For every $\varkappa\ge 1$, $c>0$, there exists a constant $A=A(\varkappa, c)\ge 2$
such that for any $(\varkappa , c)$-quasigeodesic $p$ in $\G$, $\phi(p)$ is $(A, A)$-quasigeodesic.
\end{lem}

\begin{proof}
To prove part (a) it suffices to note that every component of $\phi (p)$ is a companion of some component of $p$ and two components of $p$ are connected if and only if their companions are connected by Lemma \ref{Xlength}.

For proving part (b), observe that $|\phi^{-1}(g)|_{X \cup \mathcal{H}} \le
(2S+1)|g|_{X \cup \mathcal{H}}$ for every $g \in G$, and hence $\dxh(x,y) \le (2S+1)\dxh(\phi(x),\phi(y))$
for all $x,y \in G$.  Consider any subpath $q'$ of $\phi(p)$. By definition, there is a subpath
$q$ of $p$ such that $q'$ is contained in $\phi(q)$ and $\dxh(q'_-,\phi(q)_-) \le S+1$,
$\dxh(q'_+,\phi(q)_+) \le S+1$. Therefore,
\begin{multline*} l(q') \le l(\phi(q)) \le \max\{S,3\} l(q) \le \max\{S,3\} (\varkappa \dxh(q_-,q_+)+c) \le
\\ \max\{S,3\} \varkappa (2S+1)\dxh(\phi(q)_-,\phi(q)_+) + \max\{S,3\}c \le A\dxh(q'_-,q'_+) +A,
\end{multline*} where $A= \max\{S,3\} \varkappa (2S+1)(2S+2)+\max\{S,3\}c$.
\end{proof}

\begin{defn}[{\bf Fine geodesics}] Let $E$ be a non-negative real number. A geodesic $p$ in $\G $ will
be called $E$-{\it fine} if no component $e$ of $p$, with $l_X(e) > E$,
is connected to its companion.
%An element $g\in G$ is $E$-{\it fine} if at least one geodesic in $\G $ from $1$ to $g$ is $E$-fine.
\end{defn}

An easy argument (see Lemma \ref{lem:conn_to_comp->close_to_fix} in Section \ref{sec:main-proof}) shows that if a component of a geodesic segments $[x,y]$,
with $x,y \in \F$, is connected to its companion, then its endpoints are close to $\F$. Therefore the rest this section is devoted to studying properties of $E$-fine
geodesics.

\begin{lem} \label{lem:E-fine_close_E'-fine} Consider arbitrary $E_0\ge0$ and $\mu \ge 0$, and set
$E=E_0+2\e$ where $\e=\e(1,0,\mu)$ is given by Lemma \ref{BCP}. Suppose $p$ and $q$ are two geodesic segments
in $\G$ with $\dx(p_-,q_-)\le \mu$ and $\dx(p_+,q_+) \le \mu$, and $p$ is $E_0$-fine. Then $q$ is $E$-fine.
\end{lem}

\begin{proof} Recall that a component of a geodesic path in $\G$ is always a single edge.
Assume, on the contrary, that there is a component $e$ of $q$
that is connected to its companion $e_\phi$ and $\lx(e) >E$. Since $E>\e$, part (a) of Lemma \ref{BCP}
implies that $e$ is connected to some component $f$ of $p$, and part (b) together with the
triangle inequality yield $\lx(f) \ge \lx(e)-2\e>E_0$. Hence, according to the assumptions,
$f$ is not connected with $f_\phi$. On the other hand, $f_\phi$ and $e_\phi$ are connected by Lemma \ref{Xlength},
and since ``connectedness'' is a symmetric and transitive relation,
one can conclude that $f$ must be connected with $f_\phi$,
arriving to a contradiction. Thus $q$ is $E$-fine.
\end{proof}

The following lemma establishes a sort of local finiteness for $E$-fine geodesics.
\begin{lem}\label{cascades}
For every $E\ge 0$ there exists an increasing function
$\alpha \colon \mathbb N \cup \{0\} \to \mathbb (1,+\infty)$ such that the following holds.
Let $p$ be an $E$-fine geodesic in $\G $ such that $p_-, p_+\in \F $. Suppose that $p=p_1ep_2$,
where $e$ is a component of $p$. Then $\lx (e)\le \alpha (l(p_1))$.
\end{lem}

\begin{proof}
We will establish the claim by induction on $n=l(p_1)$.
Note that $\phi (p)$ is an
$(A,A)$-quasigeodesic without backtracking, where $A=A(1,0)\ge 2$ is the constant from Lemma \ref{qgimage}.
%(after increasing $A$, if necessary, we can assume that $A \ge 2$).
Let $\e=\e(A,A,0) \ge 0$ be the constant provided by Lemma \ref{BCP}.
Set $$\alpha(n)=(E+\e+2)(S+2\e+2)^{n+1} ~\mbox{for all } n \in \N\cup\{0\}.$$
If $\lx(e) \le \e +E$ then $\lx (e)\le \alpha (n)$ for each $n \in \N \cup \{0\}$ and the claim holds. Otherwise,
by Lemma \ref{BCP}, $e$ must be connected with a component $c$ of $\phi(p)$ and $c \neq e_\phi$ by the assumptions.
Thus there are two cases to consider.

{\it Case 1.} $c$ is a component of $\phi (p_2)$ (in particular, this happens when $n=0$). Let $s$ be an edge connecting $e_-$ to $c_-$ and let $t$ denote the segment of
$\phi (p)$ from $p_-=\phi(p)_-$ to $c_-$. Note that the path $q=p_1s$ is a $(1,2)$-quasigeodesic
in $\G$ (and hence it is an $(A,A)$-quasigeodesic) without backtracking, and the edge $e_\phi$ is a
component of $t$. If $\lx(e_\phi) \le \e$, then
$\lx (e)\le S(\e +2)$ by Lemma \ref{Xlength}; consequently $\lx (e)\le \alpha(n)$ for all
$n \in \N\cup\{0\}$ .

Suppose, now, that $\lx(e_\phi) > \e$, then Lemma \ref{BCP}, applied to the quasigeodesics
$q$ and $t$, implies that $e_\phi$ must be connected with some component of $q$.
And since $\phi(p)$ is without backtracking (by Lemma \ref{qgimage}) and $e_\phi \neq c$, $e_\phi$
cannot be connected to $s$. Hence $e_\phi$ is connected with a component $h$ of $p_1$.
In particular, $l(p_1)>0$, i.e., the base
of induction ($n=0$) has already been established. By the induction hypothesis we have $\lx(h) \le \alpha(n-1)$. On the other
hand, $\max \{ \dx ((e_\phi)_-, h_-), \dx ((e_\phi)_+, h_+)\} \le \e$ by the second part of Lemma \ref{BCP},
and the triangle inequality gives $\lx (e_\phi)\le \lx (h) + 2\e$.
Combining these with the claim of Lemma \ref{Xlength}, we obtain
\begin{multline*}\lx (e)\le S(\lx (e_\phi) + 2) \le S(\lx (h) + 2\e +2) \le S(\alpha(n-1)+2\e+2) = \\
(S+2\e+2) \alpha(n-1)-(2\e+2)(\alpha(n-1)-S)<\alpha(n).\end{multline*}

{\it Case 2.} $c$ is a component of $\phi (p_1)$. Note that $\max \{ \dx (e_-, c_-), \dx (e_+, c_+)\} \le \e$ by Lemma \ref{BCP},
therefore $\lx (e)\le \lx (c) + 2\e$ by the triangle inequality. Since every component of $\phi (p)$ is the companion of some component of $p$, $c=f_\phi $ for some component  $f$ of $p_1$. By induction, $\lx (f)\le \alpha (n-1)$. Hence, recalling
the statement of part (a) of Lemma \ref{Xlength}, we obtain
$$\lx (e)\le \lx (f_\phi) + 2\e\le S\lx (f) +2\e+2 \le (S+2\e +2)\alpha (n-1)=\alpha(n). $$
Thus we have established the inductive step and finished the proof of the lemma.
\end{proof}

\begin{defn}[{\bf The set $\mathcal{I}(x,E,R)$}] Given $x \in \F$, $E>0$ and $R\ge0$, let $\mathcal{I}(x,E,R)$
denote the set of all geodesics $p$ in $\G$ of length at most $R$ that are initial segments of
$E$-fine geodesic paths connecting $x$ with elements of $\F$.
\end{defn}

The following corollary is an immediate consequence of Lemma \ref{cascades}:

\begin{cor}\label{cor:loc_fin} For any $x \in \F$, $E \ge 0$ and $R \ge 0$,
the set $\mathcal{I}(x,E,R)$ is finite. In particular, there exists $C=C(E,R) \ge 0$ such that for any $x\in \F$ and any
$p \in \mathcal{I}(x,E,R)$ one has $\lx(p) \le C$.
\end{cor}

(The fact that $C$ does not depend on $x \in \F$ follows from the fact that left translation by $x$ is a
label-preserving automorphism of $\G$).

\begin{defn}[{\bf Large central component}] \label{def:large_centr_comp}
Consider any non-negative real number $E$, and let
$\Delta $ be a geodesic triangle in $\G $ with sides $p_1,p_2,p_3$. We will say that
$\Delta $ has an $E$-{\it large central component} if for each $i=1,2,3$, $p_i$ contains a component $a_i$,
$a_1, a_2, a_3$ are pairwise connected and \begin{equation}\label{lcc}
\lx (a_i) > T= \max \{S(3\e_0+2),E\},
\end{equation}
where $S$ is given by \eqref{eqS}, and $\e_0=\e (A,A,0) \ge 0$, $A=A(1,0) \ge 2$ are the constants
from Lemmas \ref{BCP} and \ref{qgimage}, respectively. The edges $a_1,a_2,a_3$ will be called the
\textit{sides} of the $E$-large central component.
\end{defn}

\begin{lem}\label{hren}
Let $E$ be a non-negative real number and $\Delta$ be a geodesic triangle in $\G $ with an $E$-large central
component. Suppose that vertices of $\Delta $ belong to $\F $. Then no side of $\Delta $ is $E$-fine;
more precisely, every side of the $E$-large central component is connected to its companion.
\end{lem}

\begin{proof}
Let $p_1,p_2,p_3$ denote the sides of $\Delta $ (such that $(p_{i+1})_-=(p_i)_+$, where the indices
are taken modulo $3$), and let $a_1,a_2,a_3$ be pairwise connected components of $p_1,p_2,p_3$,
respectively,  satisfying (\ref{lcc}). By part (b) of Lemma \ref{Xlength} the component
$(a_i)_\phi, (a_2)_\phi, (a_3)_\phi$  of $\phi (p_1), \phi (p_2), \phi (p_3)$ are also pairwise connected.
Note also that (\ref{lcc}) and part (a) of
Lemma \ref{Xlength} imply that $\lx ((a_i)_\phi )> 3\e_0 $ for $i=1,2,3$. Hence by
Lemmas \ref{qgimage} and \ref{BCP}, $(a_i)_\phi$ must be
connected to a component $b_i$ of $p_i$ for each $i=1,2,3$.
If $b_i=a_i$ for some $i$, then $b_i=a_i$ for all $i$ since all $b_i$ are connected
and every component of a side of $\Delta $ is isolated in that side (because the side is geodesic).
Therefore no side of $\Delta$ would be $E$-fine.
So we can assume that no $b_i$ coincides with $a_i$. Below we show that this case is
impossible by arriving at a contradiction.

\begin{figure}[!ht]
\begin{center}
   \input{fig1-n.pstex_t}
  \end{center}

\caption{} \label{fig:1}
\end{figure}

For each $i$, let $p_i=q_ia_ir_i$ and let $t_i$ denote the path (of length at most $1$)
of $\G$ connecting $(a_{i})_-$ to  $(a_{i-1})_+$ (here and below indices are modulo $3$).
The triangle $\Delta $ is cut into a hexagon and triangles
$\Sigma_i=q_it_{i}r_{i-1}$, $i=1,2,3$ (see Fig. \ref{fig:1}).
By our assumption, every component $b_i$ belongs to one of these triangles.
Hence at least one of the triangles, say $\Sigma_i$, contains exactly one of such components, say, $b_j$.
Again, since every component of a side of $\Delta $ is isolated in that side, $b_j$ is an isolated component of $\Sigma_i$. Note that $q_it_i$ and $r_{i-1}^{-1}$ are $(1,2)$-quasigeodesic paths without backtracking and
with same endpoints, hence  Lemma \ref{BCP} implies that $\lx(b_j) \le \e_0$.
However using part (2) of  Lemma \ref{BCP} and the triangle inequality %and (\ref{lcc})
we obtain  $\lx (b_j)\ge \lx ((a_j)_\phi )-2\e_0 > \e_0$,  resulting in a contradiction.
\end{proof}

\begin{defn}[{\bf Projections}]
Let $L$ be a geodesic in $\G$ and $z \in G$ be any vertex. The
\textit{projection} ${\bf pr}_{L} (z)$ of $z$ to  $L$ is the set of vertices defined by
$${\bf pr}_{L} (z)=\{g \in G \cap L\, |\, \dxh(z,g)=\dxh(z,L)\}.$$
\end{defn}

The next statement is quite standard.

\begin{lem} \label{lem:(3,0)-qg} Consider three vertices $x,y,z \in G$ in $\G$, a geodesic
segment $[x,y]$ between $x$ and $y$, any $g\in {\bf pr}_{[x,y]} (z)$ and any geodesic $p$ connecting
$z$ with $g$. Let $[x,y]=qr$ where $q$ and $r$ are geodesic subpaths with $q_-=x$, $q_+=g=r_-$, $r_+=y$.
Then $pq^{-1}$ and $pr$ are $(3,0)$-quasigeodesic paths without backtracking.
\end{lem}

\begin{proof} We will prove the statement for the path $pq^{-1}$ as the other case is symmetric. Consider any
subpath $t$ of $pq^{-1}$. The situations when $t_-,t_+$ both belong either to $p$ or to $q^{-1}$ are trivial,
therefore we can assume that $t_- \in p$ and $t_+ \in q^{-1}$.
%Let $[t_-,g]$ and $[g,t_+]$ be the segments of $t$between the corresponding vertices.
Since $t_+ \in [x,y]$ and $g\in {\bf pr}_{[x,y]} (z)$, we have
$\dxh(z,g) \le \dxh(z,t_+)$. As $t_- \in p$ and $p$ is geodesic we also have $\dxh(t_-,g)=\dxh(z,g)-\dxh(z,t_-)$.
The triangle inequality gives $\dxh(z,t_+) - \dxh(z,t_-)\le \dxh(t_-,t_+) $.
Combining these inequalities together, we can
conclude that $\dxh(t_-,g) \le \dxh(t_-,t_+)$. Therefore, applying the triangle
inequality again, we achieve
$l(t)=\dxh(t_-,g)+\dxh(g,t_+) \le 2\dxh(t_-,g)+\dxh(t_-,t_+) \le 3\dxh(t_-,t_+)$,
as required.

The paths $p$ and $q^{-1}$ are geodesic, and, hence, are without backtracking. Now, suppose a component $s$
of $p$ is connected to a component $s'$ of $q^{-1}$. Then $\dxh(s_-,s'_+) \le 1$, but
$\dxh(s_-,g) \le \dxh(s_-,s'_+)$ as shown in the previous paragraph. Therefore
$g=s_+$ and thus $\dxh(s'_+,g) \le 1$, implying that $g=s'_-$. Consequently $ss'$ is a single component of
$pq^{-1}$, and so $pq^{-1}$ is without backtracking.
\end{proof}

\begin{lem}\label{proj}
For every $E \ge 0$ there exists a constant $\eta=\eta(E)\ge 0 $ such that the following holds.
Let $\Delta =\Delta (x,y,z)$ be a geodesic triangle in $\G $ which does not contain an $E$-large central component.
%Suppose that $x,y,z\in \F $.
Then there exists a vertex $u\in {\bf pr}_{[x,y]} (z)$ (where $[x,y]$ denotes the
corresponding side of $\Delta$) and vertices $v\in [x,z]$, $w\in [y,z]$
such that $\max\{ \dx (u,v), \, \dx (u,w)\} \le\eta $.
\end{lem}

\begin{proof} Let $T>0$ be the number from Definition \ref{def:large_centr_comp}, and let $\e=\e(3,0,0) \ge 0$
and $\nu=\nu(3,0,0) \ge 0$ be the constants given by Lemmas \ref{BCP} and \ref{lem:qg-close} respectively.
Set $\eta = T+2 \e + \nu$.

Consider any vertex $g \in {\bf pr}_{[x,y]} (z)$ and any geodesic path $p$ with $p_-=z$, $p_+=g$.
Then $[x,y]$ splits in the union of two geodesics $q$ and $r$ with $q_-=x$, $q_+=g=r_-$ and $r_+=y$.
the paths $pq^{-1}$ and $pr$ are $(3,0)$-quasigeodesic without backtracking by Lemma \ref{lem:(3,0)-qg}, thus
if $g$ is a phase vertex of each of them, then we can take $u=g$ and it will be $\nu$-close to
each of the sides of $\Delta$ by Lemma \ref{lem:qg-close}.
Therefore we can suppose that $g$ is not a phase vertex of $pq^{-1}$ (the other
situation is similar). In other words, $p$ ends with an edge $e_1$ and $q^{-1}$ starts with an edge $e_2$,
such that $e_1$ and $e_2$ are $H_\lambda$-components of $p$ and $q^{-1}$ respectively for some
$\lambda \in \Lambda$ (note that in this case $(e_2)_+ \in {\bf pr}_{[x,y]} (z)$
because $\dxh((e_1)_-,(e_2)_+) \le 1 =l(e_1)$).
This implies that $e_1$ is an $H_\lambda$-component of $pr$ (since
$[x,y]$ is geodesic, all of its components consist of single edges), and thus $(e_1)_-$ and $g=(e_1)_+$ are phase
vertices of $pr$. By Lemmas \ref{lem:(3,0)-qg} and \ref{lem:qg-close},
there are vertices $w_1,w_2 \in [y,z]$
with $\dx((e_1)_-,w_1)\le \nu$ and $\dx(g,w_2) \le \nu$. The same lemmas applied to the path $pq^{-1}$ imply
that there exist $v_1, v_2 \in [x,z]$ such that $\dx((e_1)_-,v_1) \le \nu$ and $\dx((e_2)_+,v_2) \le \nu$
(see Fig. \ref{fig:2}).

\begin{figure}[!ht]
\begin{center}
   \input{fig2.pstex_t}

%Figure \ref{fig:2}.
\end{center}
\caption{} \label{fig:2}
\end{figure}

Now, if $\lx(e_1) \le T+2\e$, then we have $\dx(g,v_1) \le \lx(e_1)+\dx((e_1)_-,v_1)\le T+2\e+\nu=\eta$ and
$\dx(g,w_2) \le \nu \le \eta$. That is, we can take $u=g$, $v=v_1$ and $w=w_2$.
Similarly, if $\lx(e_2) \le T+2\e$, then we take $u=g$, $v=v_2$ and $w=w_2$. Finally, if
$\lx(e_1e_2)=\dx((e_1)_-,(e_2)_+) \le T+2\e$, we can choose $u=(e_2)_+$, $v=v_2$ and $w=w_1$.

It remains to consider the last case when $\min\{\lx(e_1),\lx(e_2),\lx(e_1e_2)\}>T+2\e$. But then we can apply
Lemma \ref{BCP} to the pair of $(3,0)$-quasigeodesic paths $pq^{-1}$ and $[z,x]$, as well as to the
pair of $(3,0)$-quasigeodesic paths $pr$ and $[z,y]$, to find a component $e_3$ on $[z,x]$ which is connected
with the component $e_1e_2$ of $pq^{-1}$, and a component $e_4$ of $[z,y]$ connected with $e_1$. Moreover,
using part (b) of this lemma together with the triangle inequality one can deduce that
$\lx(e_3) \ge \lx(e_1e_2)-2\e >T$ and $\lx(e_4) \ge \lx(e_1)-2\e >T$. Therefore the edges
$e_2,e_3,e_4$ form an $E$-large central component in $\Delta$, which contradicts our assumptions.
\end{proof}

The next statement was proved in \cite{Neumann}:

\begin{lem} \label{lem:N-app} If two points are at most distance $K$ from all three sides of a
given geodesic triangle in a geodesic metric space, then they are at distance at most $4K$ apart.
\end{lem}

The following lemma shows that the vertex $u$, given by Lemma \ref{proj}, is close to its $\phi$-image; the latter, in its turn (see Lemma \ref{lem:N-3}),
implies that $u$ is not far from $\F$.

\begin{lem}\label{lem:pt_close_to_sides->close_to_image}
Consider any real numbers $E\ge 0$ and $\eta \ge 0$. Then there is a number
$\theta=\theta(E,\eta) \ge 0$ such that the following holds.
Suppose that $x,y,z \in \F$, $\Delta=\Delta(x,y,z)$ is a geodesic triangle in $\G$
without $E$-large central component, and
$u \in G$ is a vertex that is $\eta$-close to each of the sides of $\Delta$ in the metric $\dx(\cdot,\cdot)$.
Then $\dx(u,\phi(u))\le \theta$.
\end{lem}

\begin{proof} Let $S$, $A=A(1,0)$, $\nu=\nu(A,A,0)$ and $T$ be the constants from \eqref{eqS},
Lemma \ref{qgimage}, Lemma \ref{lem:qg-close} and Definition \ref{def:large_centr_comp} respectively.
Denote $K=\max\{\eta,S\eta+\nu\}$ and let $\e=\e(1,0,K)$ be given by Lemma \ref{BCP}.
We now define $\theta=4K(T + 2\e+1)$.

By the assumptions, there are vertices $v_1 \in p_1=[x,y]$, $v_2 \in p_2=[y,z]$ and $v_3 \in p_3=[z,x]$
with $\dx(u,v_i)\le \eta$ for each $i=1,2,3$. Consequently $\dx(\phi(u),\phi(v_i)) \le S \eta$ and
$\phi(p_i)$ is an $(A,A)$-quasigeodesic without backtracking with the same endpoints as $p_i$
(by Lemma \ref{qgimage}), $i=1,2,3$.
Moreover, by Remark \ref{rem:comp_of_im_of_geod} every component of $\phi(p_i)$ is a single edge, hence
every its vertex is phase,
and so there is a vertex $w_i \in p_i$ such  $\dx(\phi(v_i),w_i) \le \nu$
(according to Lemma \ref{lem:qg-close}), $i=1,2,3$. Therefore $\dx(\phi(u),w_i) \le S\eta+\nu$, $i=1,2,3$,
and both $u$ and $\phi(u)$ are $K$-close to each of the sides of the triangle $\Delta$ in the metric
$\dx(\cdot,\cdot)$, and, hence, also in the metric $\dxh(\cdot,\cdot)$. Since the triangle $\Delta$
is geodesic with respect to the latter metric, we can apply
Lemma \ref{lem:N-app} to conclude that $\dxh(u,\phi(u)) \le 4K$.

For each $i=1,2,3$, let $q_i$ be the segment of $p_i$ (or $p_i^{-1}$) starting at $v_i$ and ending at $w_i$.
Choose any geodesic $r$ between $u$ and $\phi(u)$ in $\G$.
We need to consider two cases.

{\it Case 1.} Suppose that $r$ contains a component $c$ with $\lx(c) > T+2\e$. Since
$\dx(r_-,(q_i)_-) \le K$ and $\dx(r_+,(q_i)_+) \le K$, part (a) of Lemma \ref{BCP} tells us that $c$ must be
connected with a component $a_i$ of $q_i$, and part (b) yields $\lx(a_i) \ge \lx(c)-2\e>T$, for every $i=1,2,3$.
Therefore the triple $a_1,a_2,a_3$ forms an $E$-large central component of $\Delta$, contradicting the assumptions.
Hence Case 1 is impossible.

{\it Case 2.} Every component $c$ of $r$ satisfies $\lx(c) \le T+2\e$. Recalling that
$l(r)=\dxh(u,\phi(u)) \le 4K$, the triangle inequality implies that $\dx(u,\phi(u))\le 4K(T + 2\e+1)=\theta$.

Thus the lemma is proved.
\end{proof}

%%%%%%%%%%%%%%%%%%%%%%%%%%%%%%%%%%%%%%%%%%%%%%%%%%%%%%%%%%%%%%%%%%%%%%%%%%%%%%%%%%%%%%%%%%%%%%%%%%%%%%%%%%%%%%%%%%%%%

\section{Relative quasiconvexity of the fixed subgroup} \label{sec:main-proof}

%%%%%%%%%%%%%%%%%%%%%%%%%%%%%%%%%%%%%%%%%%%%%%%%%%%%%%%%%%%%%%%%%%%%%%%%%%%%%%%%%%%%%%%%%%%%%%%%%%%%%%%%%%%%%%%%%%%%%

The first auxiliary goal of this section is to show that the subgroup $\F$ can be
generated by a subset of $G$ of bounded diameter with respect to the metric $\dxh$.

In \cite[Thm. 1.7]{Osi06} it was shown that the relative Cayley graph $\G$ is a hyperbolic space
(in the sense of Gromov \cite{Gro}). Therefore we can use the following statement, which was proved
by Neumann in \cite[Lemma 1]{Neumann}:

\begin{lem} \label{lem:N-1} There exists $\rho \ge 0$ such that for arbitrary vertices $a,b,x,y$ of $\G$,
any geodesic segment $L$ between $x$ and $y$ and any vertices $a' \in {\bf pr}_{L}(a)$,
$b' \in {\bf pr}_{L}(a)$ we have $\dxh(a',b') \le \dxh(a,b)+\rho$.
\end{lem}

The following observation was also made in  \cite[Lemma 3]{Neumann} (it relies on the fact that
the generating set $X$ of $G$ is finite):

\begin{lem} \label{lem:N-3} For any $\theta \ge 0$ there is a constant $\mu=\mu(\theta)\ge 0$ such that
if $y \in G$
satisfies $\dx(y,\phi(y)) \le \theta$, then there exists  $y' \in \F$ with $\dx(y,y') \le \mu$.
\end{lem}

\begin{lem}\label{lem:conn_to_comp->close_to_fix} There exists a constant $\mu \ge 0$ such that for
an arbitrary geodesic segment $[x,y]$ in $\G$, with $x,y \in \F$, if a component $e$ of $[x,y]$
is connected to its companion $e_\phi$ then $\dx(e_-,\F) \le \mu$ and $\dx(e_+,\F) \le \mu$.
\end{lem}

\begin{proof} By part (b) of Lemma \ref{BCP}, applied to the paths
$[x,y]$ and $\phi([x,y])$, we have $\dx(e_-,(e_\phi)_-)\le \e$ and
$\dx(e_+,(e_\phi)_+)\le \e$, where $\e=\e(A,A,0)$ and $A=A(1,0)$ are the constants from Lemmas
\ref{BCP} and \ref{qgimage} respectively. From the definition of $e_\phi$ we see that
$\dx(\phi(e_-),(e_{\phi})_-)\le 1$ and  $\dx(\phi(e_+),(e_{\phi})_+)\le 1$, hence
$\dx(e_-,\phi(e_-))\le \e+1$ and  $\dx(e_+,\phi(e_+))\le \e+1$. Therefore $\dx(e_-,\F) \le \mu$
and $\dx(e_+,\F)\le \mu$, where $\mu=\mu(\e+1)$ is given by Lemma \ref{lem:N-3}.
\end{proof}

The next statement essentially shows that every sufficiently long segment of an $E$-fine geodesic (with endpoints from $\F$) contains a vertex
which is close to a vertex from $\F$ in the metric $\dx$. This will be important for establishing the bounded generation of $\F$ in Lemma \ref{lem:bound_gen} below.

\begin{lem} \label{lem:fp_close_to_fine} For every $E \ge 0$ there is $\xi=\xi(E) \ge 0$
such that for all $R \ge 0$ there is $\zeta=\zeta(E,R) \ge 0$ satisfying the following statement.
If $[x,y]$ is an $E$-fine geodesic
path in $\G$  with $x,y \in \F$ and $l([x,y]) \ge \zeta$, then there exists a vertex $u \in [x,y]$ such that
$\dxh(x,u) \ge R$, $\dxh(u,y) \ge R$ and $\dx(u,\F)\le \xi$.
\end{lem}

\begin{proof} Choose $\eta=\eta(E)$, $\theta=\theta(E,\eta)$ and $\xi=\mu(\theta)$ according to
Lemmas \ref{proj}, \ref{lem:pt_close_to_sides->close_to_image} and \ref{lem:N-3} respectively. For every
path $p \in \mathcal{I}(1,E,R)$ (defined above Corollary \ref{cor:loc_fin}),
choose an element $g_p \in \F$ and a geodesic $\hat p$ between $1$ and $g_p$
such that $p$ is an initial segment of $\hat p$. Since the set $\mathcal{I}(1,E,R)$ is finite
(Corollary \ref{cor:loc_fin}), we can define $\zeta=\max\{l(\hat p)\,|\, p \in \mathcal{I}(1,E,R)\}+R+\rho$,
where $\rho \ge 0$ is the constant from Lemma \ref{lem:N-1}.

Now, consider any $E$-fine geodesic path $[x,y]$ in $\G$  with $x,y \in \F$ and $l([x,y]) \ge \zeta$. Without
loss of generality we can assume that $x=1\in G$ (we can always apply the left translation by $x^{-1}$
to reduce
the situation to this case). Let $p$ be the initial segment of $[1,y]$ of length $R$ and let $g_p$, $\hat p$
be as above. Consider a geodesic triangle $\Delta(1,y,g_p)$ with the sides $[1,y]$ and $\hat p$.
Since the path $[1,y]$ is $E$-fine, $\Delta$ does not contain an $E$-large central component
(by Lemma \ref{hren}), therefore
we can apply Lemmas \ref{proj} and \ref{lem:pt_close_to_sides->close_to_image} to
find $u \in {\bf pr}_{[1,y]}(g_p)$ such that $\dx(u,\phi(u)) \le \theta$. Hence $\dx(u,\F) \le \xi$ by
Lemma \ref{lem:N-3}.

Observe that the geodesic paths $[1,y]$ and $\hat p$ have a common initial subpath $p$ of length $R$.
It follows that for any vertex $v \in {\bf pr}_{[1,y]}(g_p)$, $\dxh(g_p,v) \le \dxh(g_p,p_+)$, and so
$v$ must lie between $p_+$ and $y$ on $[1,y]$. Thus $\dxh(1,u) \ge l(p)=R$. On the other hand,
Lemma \ref{lem:N-1} tells us that $\dxh(1,u) \le \dxh(1,g_p)+\rho=l(\hat p)+\rho \le \zeta-R$.
Finally, this yields $\dxh(u,y)=\dxh(1,y)-\dxh(1,u) \ge \zeta-(\zeta-R)=R$, which concludes the proof.
\end{proof}

We are now ready to establish the bounded generation of the fixed subgroup.
\begin{lem} \label{lem:bound_gen} There exists a number $P \ge 0$ such that the subgroup $\F$ is
generated by the set $\{g \in \F\,|\, |g|_{X\cup\mathcal{H}} \le P\}$.
\end{lem}

\begin{proof} Let $\mu \ge 0$ be the constant from Lemma \ref{lem:conn_to_comp->close_to_fix}, set
$E_0=0$ and let $E \ge 0$ be the corresponding constant from Lemma \ref{lem:E-fine_close_E'-fine}.
Let $\xi=\xi(E) \ge 0$ be from Lemma \ref{lem:fp_close_to_fine}, take any real number $R>2\mu+\xi$ and
let $\zeta=\zeta(E,R)$ be given by Lemma \ref{lem:fp_close_to_fine}. Now, we define $P=\zeta+4\mu+1$
and claim that $\F$ is generated by the bounded set $\{g \in \F\,|\, |g|_{X\cup\mathcal{H}} \le P\}$.

Indeed, consider any element $w \in \F$ with $|w|_{X\cup\mathcal{H}} > P$. We will show that $w$
is a product of shorter elements from $\F$ (with respect to the generating set $X \cup\mathcal{H}$).
Take any geodesic path $p$ connecting $1$ with $w$ in $\G$, and let $v$ be a vertex of $p$
which lies within distance $1/2$ from its midpoint (so that $\dxh(1,v) \le (l(p)+1)/2$ and
$\dxh(v,w) \le (l(p)+1)/2$). Let $q$ be the maximal (simplicial)
$0$-fine subpath of $p$ containing $v$ (in general $q$ could consist of the single vertex $v$).

\textit{Case 1.} Suppose that $\dxh(1,q_-) >\mu$. Since $q$ is maximal, we see that the edge $e$ of $p$
preceding $q$ (with $e_+=q_-$) must be a component connected with its companion,
hence there exists $x\in \F$
such that $\dx(q_-,x) \le \mu$ by Lemma \ref{lem:conn_to_comp->close_to_fix}. Evidently
$\dxh(q_-,x)\le \dx(q_-,x)\le \mu$.
Since $v$ belongs to $q$ we have $\dxh(1,q_-) \le \dxh(1,v)$, and using the triangle inequality we see that
$$|x|_{X\cup\mathcal{H}}=\dxh(1,x)\le \dxh(1,q_-) + \dxh(q_-,x) \le (l(p)+1)/2 +\mu
< l(p)=|w|_{X\cup\mathcal{H}},$$ because $|w|_{X\cup\mathcal{H}}>P\ge 2\mu+1$. On the other hand,
recalling that $p$ is a geodesic path and $\dxh(1,q_-) >\mu$, we obtain
$$|x^{-1}w|_{X\cup\mathcal{H}}=\dxh(x,w) \le \mu+\dxh(q_-,w) \le \mu+l(p)-\dxh(1,q_-)
<l(p)=|w|_{X\cup\mathcal{H}}.$$ Thus we have shown that $w=x (x^{-1}w)$, where $x,x^{-1}w \in \F$
and each of these elements is strictly shorter than $w$.

\textit{Case 2.} Suppose that $\dxh(q_+,w) >\mu$. This case can be treated similarly to Case 1.

\textit{Case 3.} Suppose that $\dxh(1,q_-) \le \mu$ and $\dxh(q_+,w) \le \mu$. Again, since $q$ is maximal
and the endpoints of $p$ are in $\F$, we can use Lemma \ref{lem:conn_to_comp->close_to_fix}
to find $x,y \in \F$ such that $\dx(q_-,x)\le \mu$ and $\dx(q_+,y)\le \mu$. Since the geodesic path
$q$ is $0$-fine, any geodesic $[x,y]$ from $x$ to $y$ is $E$-fine (Lemma \ref{lem:E-fine_close_E'-fine}).
Moreover, by the triangle inequality, $l([x,y]) \ge l(q)-2\mu \ge l(p)-4\mu>P-4\mu>\zeta $, hence we
can apply Lemma \ref{lem:fp_close_to_fine}, to find a vertex $u$ on $[x,y]$ and an element $z \in \F$
such that $\dxh(x,u) \ge R$, $\dxh(u,y) \ge R$ and $\dx(u,z) \le \xi$ (see Figure \ref{fig:3}).

\begin{figure}[!ht]
\begin{center}
   \input{fig3.pstex_t}

\end{center}
\caption{} \label{fig:3}
\end{figure}

We will now estimate the lengths of the elements $x,x^{-1}z,z^{-1}y$ and $y^{-1}w$ (with respect to
the generating set $X\cup \mathcal{H}$). Observe that
$|x|_{X\cup \mathcal{H}}=\dxh(1,x) \le \dxh(1,q_-)+\dxh(q_-,x) \le 2\mu<P<|w|_{X\cup \mathcal{H}}$.
And similarly, $|y^{-1}w|<|w|_{X\cup \mathcal{H}}$. On the other hand
\begin{multline*}
|x^{-1}z|_{X\cup \mathcal{H}}=\dxh(x,z) \le \dxh(x,u)+\dxh(u,z)\le \dxh(x,y)-\dxh(u,y)+\xi \le\\
\le \dxh(q_-,q_+)+2\mu-R+\xi \le \dxh(1,w)+2\mu-R+\xi<|w|_{X\cup \mathcal{H}}.
\end{multline*}
In the same way we can show that $|z^{-1}y|_{X\cup \mathcal{H}}<|w|_{X\cup \mathcal{H}}$. Therefore
we have found a decomposition $w=x(x^{-1}z)(z^{-1}y)(y^{-1}w)$, where all of the elements
$x,x^{-1}z,z^{-1}y,y^{-1}w$ belong to $\F$ and are strictly shorter than $w$.

Thus we have considered all the possible cases and proved the lemma.
\end{proof}

The proof of the main result of this paper will require one more lemma:
\begin{lem}\label{lem:qc-aux} For any $E\ge 0$
there exists $\delta \ge 0$ such that the following holds. Let $[x,y]$ be an $E$-fine geodesic segment
in $\G$ with $x,y \in \F$. Then for any vertex $v$ of $[x,y]$ we have $\dx(v,\F) \le \delta$.
\end{lem}

\begin{proof} Let $P \ge 0$ be the constant from Lemma \ref{lem:bound_gen}. Then there are elements
$z_0,\dots,z_n \in \F$ such that $z_0=x$, $z_n=y$ and $\dxh(z_{i-1},z_i) \le P$ for every $i=1,2,\dots,n$.
Since the geodesic $[x,y]$ is $E$-fine, for every $i$, any geodesic triangle $\Delta_i$
with vertices $x,y,z_i$ will not contain an $E$-large central component (by Lemma \ref{hren}). Hence
for each $i=0,\dots,n$
we can choose $u_i \in {\bf pr}_{[x,y]}(z_i)$ according to the claim of Lemma \ref{proj}, and then
combine the statements of Lemmas \ref{lem:pt_close_to_sides->close_to_image} and \ref{lem:N-3}
to conclude that there exists $u_i' \in \F$ such that
$\dx(u_i,u_i') \le \mu$, where $\mu \ge 0$ is some predetermined constant.

Consider any vertex $v$ of $[x,y]$. Evidently, there exists $i \in \{1,\dots,n\}$ such that
$v$ lies on the subpath $p_i$ of $[x,y]$ (or of $[x,y]^{-1}$) from $u_{i-1}$ to $u_i$. Choose any geodesic
$q_i$ from $u_{i-1}'$ to $u_i'$. By Lemma \ref{lem:E-fine_close_E'-fine} $q_i$ is $E'$-fine, where
$E' \ge 0$ depends only on $E$ and $\mu$, and by Lemma \ref{lem:qg-close} there exist $\nu=\nu(1,0,\mu) \ge 0$
and a vertex  $v'$ of $q_i$ such that $\dx(v,v') \le \nu$ (see Figure \ref{fig:4}).

\begin{figure}[!ht]
\begin{center}
   \input{fig4.pstex_t}

\caption{} \label{fig:4}
\end{center}
\end{figure}

Using the
triangle inequality together with Lemma \ref{lem:N-1} we achieve
\begin{multline*} l(q_i)=\dxh(u_{i-1}',u_i') \le \dxh(u_{i-1}',u_{i-1})+ \dxh(u_{i-1},u_i)+\dxh(u_i,u_i')
\le  \\ \dxh(u_{i-1},u_i) + 2\mu \le \dxh(z_{i-1},z_i)+\rho+2\mu \le P+\rho+2\mu.
\end{multline*}
Set $R=P+\rho+2\mu$ and let $C=C(E',R)$ be from Corollary \ref{cor:loc_fin}. Then
$\dx(v',u_{i-1}') \le C$ by the latter corollary, and therefore
$\dx(v,u_{i-1}') \le \delta$, where $\delta = C + \nu$.
\end{proof}

\begin{proof}[Proof of Theorem \ref{thm:main}]
Consider any geodesic path $[g,h]$ in $\G$ with $g,h \in \F$, and let $w$ be any vertex on it.
Let $q$ be the maximal $0$-fine subpath of $[g,h]$ containing $w$. Now we can argue as in the proof of
Lemma \ref{lem:bound_gen} to find elements $x,y \in \F$ such that $\dx(q_-,x) \le \mu$ and
$\dx(q_+,y) \le \mu$, where $\mu\ge 0$ is given by Lemma \ref{lem:conn_to_comp->close_to_fix}.
Since $q$ is $0$-fine, the path $[x,y]$ is $E$-fine for some predetermined constant $E\ge 0$ by Lemma \ref{lem:E-fine_close_E'-fine}. Let $\nu=\nu(1,0,\mu) \ge 0$ be the constant provided by Lemma \ref{lem:qg-close}.
Then, after fixing any geodesic  $[x,y]$ from $x$ to $y$, we can apply this lemma to
find a vertex $v$ of $[x,y]$ such that $\dx(w,v) \le \nu$.  And Lemma \ref{lem:qc-aux} tells us that
$\dx(v,\F) \le \delta$, for some $\delta \ge 0$ which is independent of $v$. Hence $\dx(w,\F) \le \nu+\delta$,
implying that $\F$ is relatively quasiconvex with $\sigma= \nu+\delta$.
\end{proof}

\begin{proof}[Proof of Corollary \ref{cor:all_fg}]
Recall that every non-virtually cyclic group hyperbolic relative to a collection of proper subgroups
contains a non-abelian free subgroup (cf. \cite[Prop. 6.5]{DS}). Hence every slender group is either virtually cyclic or NRH.
If $G$ is hyperbolic relative to $\Hl$ and some $H_\lambda$ is virtually cyclic, then $H_\lambda $
can be excluded from the collection of peripheral subgroups of $G$ (by \cite[Theorem 2.40]{Osi06}
or by \cite[Cor. 1.14]{DS}). Since $|\Lambda |<\infty $, we can exclude all virtually cyclic
peripheral subgroups and obtain an NRH peripheral structure. Now the claim follows from
Corollary \ref{cor1} and Remark \ref{rem:parab_fg->gp_fg}.
\end{proof}

In fact, the statement of Corollary  \ref{cor:all_fg} can be refined as follows:

\begin{cor} \label{cor:pass_to_rel} Let $G$ be a finitely generated group hyperbolic relative
to a family of NRH subgroups $\Hl$. Suppose that for every $\lambda \in \Lambda$ and each
$\psi \in Aut(H_\lambda)$ the subgroup ${\rm Fix}(\psi) \le H_\lambda$ is finitely generated [finitely presented]. Then for any $\phi \in \A$, $\F$ is also finitely generated [resp., finitely presented].
\end{cor}

\begin{proof} By Corollary \ref{cor1}, $\F$ is hyperbolic relative to the finite
collection $\mathcal{O}$ of representatives of the orbits of the natural action of
$\F$ on $\mathcal{P}^\phi$. Now, for any such representative $O \in \mathcal{O}$, we have
that $|O|=\infty$ and $O=\F \cap h^{-1} H_\lambda h$, for some $\lambda \in \Lambda$ and
$h \in G$. Since $\phi$ respects the peripheral structure of $G$ (Lemma \ref{fixps}), there exist
$\mu \in \Lambda$ and $g \in G$ such that $\phi(H_\lambda) =g^{-1} H_\mu g$. However, we have
$O \le h^{-1} H_\lambda h \cap \phi (h^{-1} H_\lambda h)$ and
$\phi (h^{-1} H_\lambda h) = f^{-1} H_\mu f$,
where $f=g\phi(h)\in G$. And since $O$ is infinite, Lemma \ref{maln} yields that $\mu=\lambda$ and
$h^{-1} H_\lambda h=f^{-1} H_\mu f=\phi (h^{-1} H_\lambda h)$. Hence $O={\rm Fix}(\psi)$, where
$\psi \in Aut(h^{-1}H_\lambda h)$ is defined as the restriction of $\phi$ to $h^{-1}H_\lambda h$.
Finally, after observing that $h^{-1}H_\lambda h \cong H_\lambda$, we can use the assumptions of
the corollary together with Remark \ref{rem:parab_fg->gp_fg} to obtain the desired statement.
\end{proof}

%%%%%%%%%%%%%%%%%%%%%%%%%%%%%%%%%%%%%%%%%%%%%%%%%%%%%%%%%%%%%%%%%%%%%%%%%%%%%%%%%%%%%%%%%%%%%%%%%%%%%%%%%%%%%%%%%%%%%

\end{document}

%% file: fig1-n.pstex_t
\begin{picture}(0,0)%
\includegraphics{fig1-n.pstex}%
\end{picture}%
\setlength{\unitlength}{3947sp}%
\begingroup\makeatletter\ifx\SetFigFont\undefined%
\gdef\SetFigFont#1#2#3#4#5{%
  \reset@font\fontsize{#1}{#2pt}%
  \fontfamily{#3}\fontseries{#4}\fontshape{#5}%
  \selectfont}%
\fi\endgroup%
\begin{picture}(3007,2310)(2024,-3242)
\put(3168,-1679){\makebox(0,0)[lb]{\smash{{\SetFigFont{10}{12.0}{\rmdefault}{\mddefault}{\updefault}{\color[rgb]{0,0,0}$b_1$}%
}}}}
\put(4237,-3131){\makebox(0,0)[lb]{\smash{{\SetFigFont{10}{12.0}{\rmdefault}{\mddefault}{\updefault}{\color[rgb]{0,0,0}$q_3$}%
}}}}
\put(3454,-3035){\makebox(0,0)[lb]{\smash{{\SetFigFont{10}{12.0}{\rmdefault}{\mddefault}{\updefault}{\color[rgb]{0,0,0}$a_3$}%
}}}}
\put(2852,-3134){\makebox(0,0)[lb]{\smash{{\SetFigFont{10}{12.0}{\rmdefault}{\mddefault}{\updefault}{\color[rgb]{0,0,0}$b_3$}%
}}}}
\put(3204,-2752){\makebox(0,0)[lb]{\smash{{\SetFigFont{10}{12.0}{\rmdefault}{\mddefault}{\updefault}{\color[rgb]{0,0,0}$t_1$}%
}}}}
\put(2470,-2686){\makebox(0,0)[lb]{\smash{{\SetFigFont{10}{12.0}{\rmdefault}{\mddefault}{\updefault}{\color[rgb]{0,0,0}$q_1$}%
}}}}
\put(4484,-2732){\makebox(0,0)[lb]{\smash{{\SetFigFont{10}{12.0}{\rmdefault}{\mddefault}{\updefault}{\color[rgb]{0,0,0}$r_2$}%
}}}}
\put(3531,-2166){\makebox(0,0)[lb]{\smash{{\SetFigFont{10}{12.0}{\rmdefault}{\mddefault}{\updefault}{\color[rgb]{0,0,0}$t_2$}%
}}}}
\put(2488,-3176){\makebox(0,0)[lb]{\smash{{\SetFigFont{10}{12.0}{\rmdefault}{\mddefault}{\updefault}{\color[rgb]{0,0,0}$r_3$}%
}}}}
\put(3774,-2684){\makebox(0,0)[lb]{\smash{{\SetFigFont{10}{12.0}{\rmdefault}{\mddefault}{\updefault}{\color[rgb]{0,0,0}$t_3$}%
}}}}
\put(4064,-2213){\makebox(0,0)[lb]{\smash{{\SetFigFont{10}{12.0}{\rmdefault}{\mddefault}{\updefault}{\color[rgb]{0,0,0}$a_2$}%
}}}}
\put(3777,-1692){\makebox(0,0)[lb]{\smash{{\SetFigFont{10}{12.0}{\rmdefault}{\mddefault}{\updefault}{\color[rgb]{0,0,0}$b_2$}%
}}}}
\put(2847,-2210){\makebox(0,0)[lb]{\smash{{\SetFigFont{10}{12.0}{\rmdefault}{\mddefault}{\updefault}{\color[rgb]{0,0,0}$a_1$}%
}}}}
\put(3252,-1438){\makebox(0,0)[lb]{\smash{{\SetFigFont{10}{12.0}{\rmdefault}{\mddefault}{\updefault}{\color[rgb]{0,0,0}$r_1$}%
}}}}
\put(3888,-1883){\makebox(0,0)[lb]{\smash{{\SetFigFont{10}{12.0}{\rmdefault}{\mddefault}{\updefault}{\color[rgb]{0,0,0}$q_2$}%
}}}}
\end{picture}%

%% file: fig2.pstex_t
\begin{picture}(0,0)%
\includegraphics{fig2.pstex}%
\end{picture}%
\setlength{\unitlength}{3947sp}%
\begingroup\makeatletter\ifx\SetFigFont\undefined%
\gdef\SetFigFont#1#2#3#4#5{%
  \reset@font\fontsize{#1}{#2pt}%
  \fontfamily{#3}\fontseries{#4}\fontshape{#5}%
  \selectfont}%
\fi\endgroup%
\begin{picture}(3591,2536)(1982,-4852)
\put(5558,-4661){\makebox(0,0)[lb]{\smash{{\SetFigFont{10}{12.0}{\rmdefault}{\mddefault}{\updefault}{\color[rgb]{0,0,0}$y$}%
}}}}
\put(3521,-4754){\makebox(0,0)[lb]{\smash{{\SetFigFont{10}{12.0}{\rmdefault}{\mddefault}{\updefault}{\color[rgb]{0,0,0}$e_2$}%
}}}}
\put(4565,-4770){\makebox(0,0)[lb]{\smash{{\SetFigFont{10}{12.0}{\rmdefault}{\mddefault}{\updefault}{\color[rgb]{0,0,0}$r$}%
}}}}
\put(3802,-4759){\makebox(0,0)[lb]{\smash{{\SetFigFont{10}{12.0}{\rmdefault}{\mddefault}{\updefault}{\color[rgb]{0,0,0}$g$}%
}}}}
\put(3887,-3630){\makebox(0,0)[lb]{\smash{{\SetFigFont{10}{12.0}{\rmdefault}{\mddefault}{\updefault}{\color[rgb]{0,0,0}$p$}%
}}}}
\put(2864,-4786){\makebox(0,0)[lb]{\smash{{\SetFigFont{10}{12.0}{\rmdefault}{\mddefault}{\updefault}{\color[rgb]{0,0,0}$q^{-1}$}%
}}}}
\put(3851,-4429){\makebox(0,0)[lb]{\smash{{\SetFigFont{10}{12.0}{\rmdefault}{\mddefault}{\updefault}{\color[rgb]{0,0,0}$e_1$}%
}}}}
\put(3781,-2468){\makebox(0,0)[lb]{\smash{{\SetFigFont{10}{12.0}{\rmdefault}{\mddefault}{\updefault}{\color[rgb]{0,0,0}$z$}%
}}}}
\put(1997,-4667){\makebox(0,0)[lb]{\smash{{\SetFigFont{10}{12.0}{\rmdefault}{\mddefault}{\updefault}{\color[rgb]{0,0,0}$x$}%
}}}}
\put(4584,-3815){\makebox(0,0)[lb]{\smash{{\SetFigFont{10}{12.0}{\rmdefault}{\mddefault}{\updefault}{\color[rgb]{0,0,0}$w_1$}%
}}}}
\put(4820,-4074){\makebox(0,0)[lb]{\smash{{\SetFigFont{10}{12.0}{\rmdefault}{\mddefault}{\updefault}{\color[rgb]{0,0,0}$w_2$}%
}}}}
\put(2765,-4004){\makebox(0,0)[lb]{\smash{{\SetFigFont{10}{12.0}{\rmdefault}{\mddefault}{\updefault}{\color[rgb]{0,0,0}$v_2$}%
}}}}
\put(3001,-3791){\makebox(0,0)[lb]{\smash{{\SetFigFont{10}{12.0}{\rmdefault}{\mddefault}{\updefault}{\color[rgb]{0,0,0}$v_1$}%
}}}}
\end{picture}%

%% file: fig3.pstex_t
\begin{picture}(0,0)%
\includegraphics{fig3.pstex}%
\end{picture}%
\setlength{\unitlength}{3986sp}%
\begingroup\makeatletter\ifx\SetFigFont\undefined%
\gdef\SetFigFont#1#2#3#4#5{%
  \reset@font\fontsize{#1}{#2pt}%
  \fontfamily{#3}\fontseries{#4}\fontshape{#5}%
  \selectfont}%
\fi\endgroup%
\begin{picture}(4511,1299)(1914,-2112)
\put(4223,-968){\makebox(0,0)[lb]{\smash{{\SetFigFont{10}{12.0}{\rmdefault}{\mddefault}{\updefault}{\color[rgb]{0,0,0}$q$}%
}}}}
\put(5770,-1754){\makebox(0,0)[lb]{\smash{{\SetFigFont{10}{12.0}{\rmdefault}{\mddefault}{\updefault}{\color[rgb]{0,0,0}$y$}%
}}}}
\put(6396,-1648){\makebox(0,0)[lb]{\smash{{\SetFigFont{10}{12.0}{\rmdefault}{\mddefault}{\updefault}{\color[rgb]{0,0,0}$w$}%
}}}}
\put(4088,-2045){\makebox(0,0)[lb]{\smash{{\SetFigFont{10}{12.0}{\rmdefault}{\mddefault}{\updefault}{\color[rgb]{0,0,0}$z$}%
}}}}
\put(4194,-1334){\makebox(0,0)[lb]{\smash{{\SetFigFont{10}{12.0}{\rmdefault}{\mddefault}{\updefault}{\color[rgb]{0,0,0}$u$}%
}}}}
\put(1929,-1633){\makebox(0,0)[lb]{\smash{{\SetFigFont{10}{12.0}{\rmdefault}{\mddefault}{\updefault}{\color[rgb]{0,0,0}$1$}%
}}}}
\put(2700,-1730){\makebox(0,0)[lb]{\smash{{\SetFigFont{10}{12.0}{\rmdefault}{\mddefault}{\updefault}{\color[rgb]{0,0,0}$x$}%
}}}}
\put(6117,-1301){\makebox(0,0)[lb]{\smash{{\SetFigFont{10}{12.0}{\rmdefault}{\mddefault}{\updefault}{\color[rgb]{0,0,0}$p$}%
}}}}
\put(5790,-1191){\makebox(0,0)[lb]{\smash{{\SetFigFont{10}{12.0}{\rmdefault}{\mddefault}{\updefault}{\color[rgb]{0,0,0}$q_+$}%
}}}}
\put(2620,-1175){\makebox(0,0)[lb]{\smash{{\SetFigFont{10}{12.0}{\rmdefault}{\mddefault}{\updefault}{\color[rgb]{0,0,0}$q_-$}%
}}}}
\end{picture}%

%% file: fig4.pstex_t
\begin{picture}(0,0)%
\includegraphics{fig4.pstex}%
\end{picture}%
\setlength{\unitlength}{3947sp}%
\begingroup\makeatletter\ifx\SetFigFont\undefined%
\gdef\SetFigFont#1#2#3#4#5{%
  \reset@font\fontsize{#1}{#2pt}%
  \fontfamily{#3}\fontseries{#4}\fontshape{#5}%
  \selectfont}%
\fi\endgroup%
\begin{picture}(3504,1605)(2025,-4346)
\put(4213,-2906){\makebox(0,0)[lb]{\smash{{\SetFigFont{10}{12.0}{\rmdefault}{\mddefault}{\updefault}{\color[rgb]{0,0,0}$z_{i}$}%
}}}}
\put(3717,-2896){\makebox(0,0)[lb]{\smash{{\SetFigFont{10}{12.0}{\rmdefault}{\mddefault}{\updefault}{\color[rgb]{0,0,0}$z_{i-1}$}%
}}}}
\put(3601,-3821){\makebox(0,0)[lb]{\smash{{\SetFigFont{10}{12.0}{\rmdefault}{\mddefault}{\updefault}{\color[rgb]{0,0,0}$u_{i-1}$}%
}}}}
\put(3982,-4279){\makebox(0,0)[lb]{\smash{{\SetFigFont{10}{12.0}{\rmdefault}{\mddefault}{\updefault}{\color[rgb]{0,0,0}$v'$}%
}}}}
\put(2040,-3922){\makebox(0,0)[lb]{\smash{{\SetFigFont{10}{12.0}{\rmdefault}{\mddefault}{\updefault}{\color[rgb]{0,0,0}$u_1=x$}%
}}}}
\put(5514,-3922){\makebox(0,0)[lb]{\smash{{\SetFigFont{10}{12.0}{\rmdefault}{\mddefault}{\updefault}{\color[rgb]{0,0,0}$y=u_n$}%
}}}}
\put(3249,-4043){\makebox(0,0)[lb]{\smash{{\SetFigFont{10}{12.0}{\rmdefault}{\mddefault}{\updefault}{\color[rgb]{0,0,0}$u_2$}%
}}}}
\put(2777,-4053){\makebox(0,0)[lb]{\smash{{\SetFigFont{10}{12.0}{\rmdefault}{\mddefault}{\updefault}{\color[rgb]{0,0,0}$u_1$}%
}}}}
\put(3587,-4226){\makebox(0,0)[lb]{\smash{{\SetFigFont{10}{12.0}{\rmdefault}{\mddefault}{\updefault}{\color[rgb]{0,0,0}$u_{i-1}'$}%
}}}}
\put(5158,-4033){\makebox(0,0)[lb]{\smash{{\SetFigFont{10}{12.0}{\rmdefault}{\mddefault}{\updefault}{\color[rgb]{0,0,0}$u_{n-1}$}%
}}}}
\put(3240,-3098){\makebox(0,0)[lb]{\smash{{\SetFigFont{10}{12.0}{\rmdefault}{\mddefault}{\updefault}{\color[rgb]{0,0,0}$z_2$}%
}}}}
\put(5143,-3436){\makebox(0,0)[lb]{\smash{{\SetFigFont{10}{12.0}{\rmdefault}{\mddefault}{\updefault}{\color[rgb]{0,0,0}$z_{n-1}$}%
}}}}
\put(2772,-3445){\makebox(0,0)[lb]{\smash{{\SetFigFont{10}{12.0}{\rmdefault}{\mddefault}{\updefault}{\color[rgb]{0,0,0}$z_1$}%
}}}}
\put(3987,-3807){\makebox(0,0)[lb]{\smash{{\SetFigFont{10}{12.0}{\rmdefault}{\mddefault}{\updefault}{\color[rgb]{0,0,0}$v$}%
}}}}
\put(4305,-3836){\makebox(0,0)[lb]{\smash{{\SetFigFont{10}{12.0}{\rmdefault}{\mddefault}{\updefault}{\color[rgb]{0,0,0}$u_{i}$}%
}}}}
\put(4306,-4234){\makebox(0,0)[lb]{\smash{{\SetFigFont{10}{12.0}{\rmdefault}{\mddefault}{\updefault}{\color[rgb]{0,0,0}$u_{i}'$}%
}}}}
\end{picture}%